\documentclass[11pt]{amsart}
\usepackage{epsfig,url}
\usepackage{mathdots}

\usepackage{epsfig,url}
\usepackage{mathdots}

\usepackage{fancybox}
\usepackage{tikz}
\usepackage{xspace}
\usepackage[colorlinks=true,
linkcolor=black,
filecolor=brown,
citecolor=black]{hyperref}

\makeatletter

\renewcommand\section{\@startsection {section}{1}{\z@}
{-30pt \@plus -1ex \@minus -.2ex}
{2.3ex \@plus.2ex}
{\normalfont\normalsize\bfseries}}

\renewcommand\subsection{\@startsection{subsection}{2}{\z@}
{-3.25ex\@plus -1ex \@minus -.2ex}
{1.5ex \@plus .2ex}
{\normalfont\normalsize\bfseries}}

\renewcommand{\@seccntformat}[1]{\csname the#1\endcsname. }

\makeatother

\newtheorem{theorem}{Theorem}[section]
\newtheorem{lemma}[theorem]{Lemma}

\newtheorem{corollary}[theorem]{Corollary}
\newtheorem{proposition}[theorem]{Proposition}
\newtheorem{property}[theorem]{Property}

\theoremstyle{definition}

\newtheorem{note}[theorem]{Note}

\theoremstyle{remark}

\begin{document}

\title{Filter integrals for  orthogonal polynomials}

\author[T. Amdeberhan et al.]{Tewodros  Amdeberhan}
\address{Department of Mathematics, Tulane University, New Orleans, LA 70118}
\email{tamdeber@tulane.edu}

\author[]{Adriana Duncan}
\address{Department of Mathematics, Tulane University, New Orleans, LA 70118}
\email{aduncan3@tulane.edu}

\author[]{Victor H. Moll}
\address{Department of Mathematics, Tulane University, New Orleans, LA 70118}
\email{vhm@tulane.edu}

\author[]{Vaishavi Sharma}
\address{Department of Mathematics, Tulane University, New Orleans, LA 70118}
\email{vsharma1@tulane.edu}


\date{\today}


\begin{abstract}
Motivated by an expression by Persson and Strang on an integral involving  Legendre polynomials, stating that 
the square of $P_{2n+1}(x)/x$ integrated over $[-1,1]$ is always $2$, we present 
analog results for  Hermite, Chebyshev, Laguerre and 
Gegenbauer polynomials as well as the original Legendre polynomial with even index. 
\end{abstract}

\maketitle

\newcommand{\ba}{\begin{eqnarray}}
\newcommand{\ea}{\end{eqnarray}}
\newcommand{\ift}{\int_{0}^{\infty}}
\newcommand{\nn}{\nonumber}
\newcommand{\no}{\noindent}
\newcommand{\lf}{\left\lfloor}
\newcommand{\rf}{\right\rfloor}
\newcommand{\realpart}{\mathop{\rm Re}\nolimits}
\newcommand{\imagpart}{\mathop{\rm Im}\nolimits}
\newcommand{\R}{X}
\newcommand{\co}{}

\newcommand{\op}[1]{\ensuremath{\operatorname{#1}}}
\newcommand{\pFq}[5]{\ensuremath{{}_{#1}F_{#2} \left( \genfrac{}{}{0pt}{}{#3}
{#4} \bigg| {#5} \right)}}

\newtheorem{Definition}{\bf Definition}[section]
\newtheorem{Thm}[Definition]{\bf Theorem}
\newtheorem{Example}[Definition]{\bf Example}
\newtheorem{Lem}[Definition]{\bf Lemma}
\newtheorem{Cor}[Definition]{\bf Corollary}
\newtheorem{Prop}[Definition]{\bf Proposition}
\numberwithin{equation}{section}

\section{Introduction}
\label{sec-intro}

Persson and Strang \cite{persson-2003a} presented a family of filters, symmetric of even degree $n$,  obtained by 
sampling a multiple of the polynomial
function $x^{-1}P_{n+1}(x)$, where $P_{n}(x)$ is the Legendre polynomial.  In their analysis of the error due to noise, they 
proved Theorem \ref{thm-int1} below. 
The authors commented that this result was not expected to them. 

\begin{theorem}
\label{thm-int1}
For all $n \in \mathbb{N}$, 
\begin{equation}
\label{legend-odd}
\int_{-1}^{1} \left( \frac{P_{2n+1}(x)}{x} \right)^{2} \, dx = 2. \quad \co
\end{equation}
\end{theorem}

The goal of this note is to present extensions of Theorem \ref{thm-int1} to other families of orthogonal polynomials
$\{ A_{n}(x) \}$, with weight $w_{A}(x)$.  The 
integrals evaluated here have the form 
\begin{equation}
I_{n}(a,b;w) = \int_{a}^{b} \left( \frac{A_{n}(x) - A_{n}(0)}{x} \right)^{2} \, w_{A}(x) dx,
\end{equation}
\noindent
where the term $A_{n}(0)$ is introduced to guarantee  convergence of the integral.

The rest of this introduction states our results. The polynomials considered here include
Legendre, Hermite, Chebyshev, Laguerre and Gegenbauer. Some basic properties of each of these polynomials are 
in the introduction to the corresponding sections. 

Section \ref{legendre-even} contains the evaluation complementary to \eqref{legend-odd} for even indexed Legendre
polynomials. The statement is given in terms of 
\begin{equation}
\beta(n) = \begin{cases} 0 & \quad \text{if } \,\, n \,\, \text{is odd} \\ 
2^{-n} \binom{n}{n/2}  & \quad \text{if } \,\, n \,\, \text{is even}.
\end{cases}
\end{equation}

\begin{theorem}
\label{thm-legendre1}
For $n \in \mathbb{N}$ and $P_{n}$ the Legendre polynomial,
\begin{equation}
\int_{-1}^{1} \left( \frac{P_{n}(x)-P_{n}(0)}{x} \right)^{2} \, dx = 2 \left[ 1 - \beta^{2}(n) \right]. \quad \co
\end{equation}
\end{theorem}

Section \ref{hermite-even} contains the analog to Theorem \ref{thm-legendre1} for Hermite polynomials $H_{n}(x)$. These 
form a family of orthogonal 
polynomials on $\mathbb{R}$  with respect to the weight $w_{H}(x) = e^{-x^{2}}$. 

\begin{theorem}
For $n \in \mathbb{N}$ and $H_{n}$ the Hermite  polynomial, then 
\begin{equation}
\int_{-\infty}^{\infty} \left( \frac{H_{n}(x)-H_{n}(0)}{x} \right)^{2} \, e^{-x^{2}}  dx = \sqrt{\pi} n! 2^{n+1}  \left[ 1 - \beta(n) \right].
\quad \co
\end{equation}
\end{theorem}

Section \ref{cheby-int1} analyzes the corresponding integral for Chebyshev polynomials of both kinds $\{ T_{n} \}$ and 
$\{U_{n} \}$. These are defined by the relations 
\begin{equation}
\cos(n x) = T_{n}(\cos x) \quad \text{ and } \quad \frac{\sin((n+1) x) }{\sin x} = U_{n}(\cos x).
\end{equation}
\noindent
The family $\{T_{n} \}$ is orthogonal with respect to the weight $(1-x^{2})^{-1/2}$ and $\{ U_{n} \}$ with respect to 
$(1-x^{2})^{1/2}$ on the interval $[-1, 1]$.

\begin{theorem}
For $n \in \mathbb{N}$, 
\begin{eqnarray*}
  \int_{-1}^{1} \left[ \frac{T_{n}(x) - T_{n}(0)}{x} \right]^{2} \frac{dx}{\sqrt{1-x^{2}}} & = &  \pi n,  \quad \co \\
 \int_{-1}^{1} \left[ \frac{U_{n}(x) - U_{n}(0)}{x} \right]^{2}  \sqrt{1-x^{2}} \, dx & = & 
\begin{cases} \pi n & \quad \text{for} \,\, n  \,\, \text{even} \quad \co \\
 \pi (n+1) & \quad \text{if} \,\, n \,\, \text{is odd}.
 \end{cases}
\end{eqnarray*}
\end{theorem}

Section \ref{lague-int1} describes the evaluation of the corresponding integrals for the family of Laguerre polynomials. These 
are orthogonal polynomials on $[0, \infty)$ with the weight function $w_{L}(x) = e^{-x}$.

\begin{theorem}
For $n \in \mathbb{N}$, let $L_{n}(x)$ be the Laguerre polynomial. Then 
\begin{equation}
\int_{0}^{\infty} \left[ \frac{L_{n}(x) - L_{n}(0)}{x} \right]^{2} e^{-x} \, dx = 2n - H_{n} \quad \co
\end{equation}
\noindent
where $H_{n}$ is the harmonic number.
\end{theorem}

Finally, Section \ref{gegenb-int1} discusses the Gegenbauer polynomials $C_{n}^{(a)}(x)$. These are defined by the
recurrence
\begin{equation}
C_{n}^{(a)}(x) = \frac{1}{n} 
\left[ 2x(n+a-1) C_{n-1}^{(a)}(x) - (n+2a-2) C_{n-2}^{(a)}(x) \right], \quad \co
\end{equation}
\noindent
with initial conditions $C_{0}^{(a)}(x) = 1$ and $C_{1}^{(a)}(x) = 2ax$. The Gegenbauer polynomials are orthogonal with 
respect to the weight $(1-x^{2})^{a-1/2}$ and are normalized by 
\begin{equation}
\int_{-1}^{1} \left[ C_{n}^{a}(x) \right]^{2} (1-x^{2})^{a-1/2} \, dx = \frac{\pi \Gamma(n+2 a)}{2^{2a-1} n! (n+a) 
\Gamma^{2}(a) }. \quad \co
\end{equation}

The corresponding result is:

\begin{theorem}
For $n \in \mathbb{N}$, let $C_{n}^{a}(x)$ be the Gegenbauer  polynomial. Then 
\begin{equation}
\gamma_{n}(a) = \int_{-1}^{1} \left[ \frac{C_{n}^{a}(x) - C_{n}^{a}(0)}{x} \right]^{2} (1-x^2)^{a-1/2}\, dx
\end{equation}
\noindent 
where 
\begin{equation}
\gamma_{2n+1}(a) = \frac{\pi \Gamma(2a+2n+1)}{2^{2a-2} (2n+1)! \Gamma^{2}(a)} \quad \co
\end{equation}
\noindent
for an odd index and, in the case of even index,
\begin{equation}
\gamma_{2n}(a) = \frac{\pi \Gamma(2a) \Gamma(a+n)}{2^{2a-3} \Gamma(2n+1) \Gamma^{3}(a)} \R_{n}(a), \quad \co
\end{equation}
\noindent
where $\R_{n}(a)$ is the polynomial
\begin{equation}
\R_{n}(a)  =   2^{2n-1} \left( a + \tfrac{1}{2} \right)_{n} - \binom{2n-1}{n-1} (a)_{n}.
\end{equation}
\end{theorem}

The last section contains some conjectures on the polynomial $\R_{n}(a)$. 

\section{The extension of Legendre integrals to even indices}
\label{legendre-even}

The proof of Theorem \ref{thm-int1} in \cite{persson-2003a} starts with the classical recurrence 
\begin{equation}
(n+1)P_{n+1}(x) - (2n+1)xP_{n}(x) + nP_{n-1}(x) = 0 \quad \co
\label{three-term} 
\end{equation}
\noindent
satisfied by the Legendre polynomials. This appears as entry $7.221.1$ in \cite{gradshteyn-2015a}. Divide \eqref{three-term} 
by $x$ and write $n = 2m$ to obtain 
\begin{equation}
\label{three-term2}
(2m+1)\frac{P_{2m+1}(x)}{x} - (4m+1)P_{2m}(x) + 2m\frac{P_{2m-1}(x)}{x}= 0. \quad \co
\end{equation}
\noindent
Since $P_{k}(0) = 0 $ for $k$ odd,  this is a polynomial relation. 

Define 
\begin{equation}
c_{m} = \int_{-1}^{1} \left( \frac{P_{2m-1}(x)}{x} \right)^{2} \, dx.
\end{equation}
\noindent
Square \eqref{three-term2} and integrate over $[-1, 1]$ to obtain 
\begin{multline*}
(2m+1)^{2}c_{m+1} = (4m+1)^{2} \int_{-1}^{1} P_{2m}^{2}(x) \, dx  \quad \co  \\ - 
4m(4m+1) \int_{-1}^{1} P_{2m}(x) \frac{P_{2m-1}(x)}{x} \, dx + 4m^{2} c_{m}.
\end{multline*}
\noindent
The integral 
\begin{equation}
\label{sq-leg1}
\int_{-1}^{1} P_{2m}^{2}(x) \, dx = \frac{2}{4m+1} \quad \co
\end{equation}
\noindent
is well-known (see entry $7.221.1$ in \cite{gradshteyn-2015a}). The other integral is zero since $P_{2m-1}(x)/x$ is a polynomial of degree $2m-2$, so it is 
orthogonal to $P_{2m}(x)$.  It follows that 
\begin{equation}
(2m+1)^{2} c_{m+1} = 2(4m+1) + 4m^{2} c_{m-1}. \quad \co
\end{equation}
\noindent
The initial value $c_{1}=2$ shows that $c_{2m+1} =  2$ for all $m \in \mathbb{N}$.

The goal of this section is to present the result given in Theorem \ref{thm-int1} for the case of an even index $n$. The 
polynomial $P_{n}$ does not vanish at $x=0$, so it is necessary to modify the integrand slightly. 

\begin{theorem}
\label{thm-int2}
For all $m \in \mathbb{N}$, 
\begin{equation}
\int_{-1}^{1} \left( \frac{P_{2m}(x)-P_{2m}(0)}{x} \right)^{2} \, dx = 2 \left[ 1 - \frac{1}{2^{4m} } \binom{2m}{m}^{2} \right].
\quad \co
\end{equation}
\end{theorem}

The proof of this assertion  begins with \eqref{three-term} for $n = 2m+1$, in the form, 
\begin{multline}
\label{recu-1}
(2m+2) \left[ P_{2m+2}(x)- P_{2m+2}(0) \right] - (4m+3) x P_{2m+1}(x)  \quad \co  \\+ 
(2m+1) \left[ P_{2m}(x) - P_{2m}(0) \right] = 0.
\end{multline}
\noindent
The contribution of the polynomials at $x=0$ vanishes since 
\begin{equation}
-(2m+2)P_{2m+2}(0) - (2m+1)P_{2m}(0) = 0 \quad \co 
\end{equation}
\noindent
in view of 
\begin{equation}
P_{2m}(0) = \frac{(-1)^{m} \binom{2m}{m}}{2^{2m}}. \quad \co 
\end{equation}
\noindent
Then \eqref{recu-1} may be written as 
\begin{multline}
\label{recu-2}
2(m+1) \left[ \frac{P_{2m+2}(x) - P_{2m+2}(0)}{x} \right] =  \quad \co \\
(4m+3)P_{2m+1}(x) - (2m+1)  \left[ \frac{P_{2m}(x) - P_{2m}(0)}{x} \right].
\end{multline}
\noindent
Squaring this relation produces 
\begin{multline}
\label{recu-3}
4(m+1)^{2} \left[ \frac{P_{2m+2}(x) - P_{2m+2}(0)}{x} \right]^{2} = \quad \co  \\
(4m+3)^{2} P_{2m+1}^{2}(x) -   \\
2(4m+3)(2m+1) P_{2m+1}(x) \times \left[ \frac{P_{2m}(x) - P_{2m}(0)}{x} \right] \\
+ (2m+1)^{2} \left[ \frac{P_{2m}(x) - P_{2m}(0)}{x} \right]^{2}. 
\end{multline}

Denote
\begin{equation}
\label{def-um}
w_{m} = \int_{-1}^{1} \left[ \frac{P_{2m}(x)-P_{2m}(0)}{x} \right]^{2} \, dx,
\end{equation}
\noindent
integrate \eqref{recu-3} from $x=-1$ to $+1$ and use \eqref{sq-leg1} 
\begin{equation}
\int_{-1}^{1} P_{2m+1}^{2}(x) \, dx = \frac{2}{4m+3}.   \quad \co 
\end{equation}
\noindent
Also apply the  relation 
\begin{equation}
\int_{-1}^{1} P_{2m+1}(x) \times \left[ \frac{P_{2m}(x) - P_{2m}(0}{x} \right] \, dx = 0 \quad \co
\end{equation}
\noindent
coming from  the orthogonality of $P_{2m+1}(x)$ and $(P_{2m}(x) - P_{2m}(0))/x$, the latter being a polynomial of 
degree $2m-1$.  This proves the next statement.

\begin{lemma}
\label{lem-2.2}
The integrals $w_{m}$ defined in \eqref{def-um} satisfy the recurrence 
\begin{equation}
4(m+1)^{2}w_{m+1} = 2(4m+3) + (2m+1)^{2} w_{m}, \quad \co
\label{recu-U1}
\end{equation}
\noindent 
with initial condition $u_{0} = 0$.
\end{lemma}

The expression for $w_{m}$ in Theorem \ref{thm-int2} follows directly from Lemma \ref{lem-2.2}. Indeed, the
 solution $w_{m}$ of 
\eqref{recu-U1} has the form $w_{m} = w_{m}^{(h)} + w_{m}^{(p)}$, where $w_{m}^{(p)}$ is a particular solution and 
$w_{m}^{(h)}$ is the general solution of the homogeneous equation 
\begin{equation}
\label{homo-equ1}
(2m+2)^2w_{m+1}^{(h)} = (2m+1)^{2}  w_{m}^{(h)}.
\end{equation}
\noindent
Since the coefficients of $w_{m+1}$ and $w_{m}$ in \eqref{recu-U1} 
have the same degree, one tries as a particular solution a constant 
$w_{m}^{(p)} =  C$. Replacing in \eqref{recu-U1} yields $C =  2$.  The general solution of the homogeneous equation 
\eqref{homo-equ1} is obtained by iterating the recurrence to obtain 
\begin{equation}
w_{m}^{(h)} =\left[  \frac{[(2m-1)(2m-3) \cdots 3 \cdot 1}{(2m)(2m-2) \cdots 4 \cdot 2 } \right]^{2} 
 = 2^{-4m} \binom{2m}{m}^{2}
\end{equation}
\noindent
up to a scalar. Therefore, the solution of \eqref{recu-U1} has the form 
\begin{equation}
w_{m} = 2 + \alpha 2^{-4m} \binom{2m}{m}^{2}.
\end{equation}
\noindent
The initial condition $w_{0} = 0$ gives $\alpha = - 2$ proving Theorem \ref{thm-int2}.

\section{The Hermite case}
\label{hermite-even}

The Hermite polynomials $\{ H_{n}(x) \}$ form a family of orthogonal polynomials on $\mathbb{R}$ with respect to a 
gaussian weight $e^{-x^{2}}$ with the normalization condition (appearing as entry $8.959(1).2$ in  \cite{gradshteyn-2015a}) 
\begin{equation}
\label{norma-her}
\int_{-\infty}^{\infty} H_{n}(x) H_{m}(x) e^{-x^{2}} \, dx = \sqrt{\pi} 2^{n} n! \delta_{nm} \quad \co
\end{equation}
\noindent
where $\delta_{nm}$ is the Kronecker delta. This family is the so-called \texttt{physicists Hermite}, a second family with 
$\tfrac{1}{2}x^{2}$ instead of $x^{2}$ in the exponent of the weight also appears in the literature. The 
latter  are the 
\texttt{probabilists Hermite}, so one should be careful in checking identities.

As part of a general procedure in extending  Theorem \ref{thm-int1},  E.~Diekema and T.~Koorwinder \cite{diekema-2011a} 
established the next statement.

\begin{theorem}
\label{theorem-1}
For all $n \in \mathbb{N}$, define 
\begin{equation}
\label{hermite-form1}
I_{2n+1} = \int_{-\infty}^{\infty} \left( \frac{H_{2n+1}(x)}{x} \right)^{2} \, e^{-x^{2}} dx.
\end{equation}
\noindent
Then 
\begin{equation}
\label{hermite-odd}
I_{2n+1} = \sqrt{\pi} 2^{2n+2} (2n+1)! \quad \co 
\end{equation}
\end{theorem}

The goal of this section is to determine the analog of \eqref{hermite-form1} for even indices.  Introduce the notation 
\begin{equation}
\label{value-herm1}
I_{2n} = \int_{-\infty}^{\infty} \left( \frac{H_{2n}(x)-H_{2n}(0)}{x} \right)^{2} \, e^{-x^{2}} dx,
\end{equation}
\noindent
complementary to \eqref{hermite-form1}.  Symbolic evaluation using \texttt{Mathematica} shows that 
\begin{equation}
O_{n} = \frac{I_{2n}}{2^{2n+1} \sqrt{\pi}}
\end{equation}
\noindent
is an odd integer sequence, starting with $1, \, 15, \, 495, \, 29295, \, 2735775$. A search in OEIS (the 
Online Encyclopedia of Integer
Sequences) produced  entry $\text{A}151816$. This is how the statement of the next theorem was obtained.

\begin{theorem}
\label{theorem-2}
For $n \in \mathbb{N}$ the integral $I_{2n}$ defined in \eqref{value-herm1}  is computed  by 
\begin{equation}
I_{2n} =2  \sqrt{\pi}  (2n)! \left[ 2^{2n} - \binom{2n}{n} \right]. \quad \co
\end{equation}
\end{theorem}

\begin{note}
The results stated in Theorems \ref{theorem-1} and \ref{theorem-2} can be combined  as 
\begin{equation}
I_{k} = 2 \sqrt{\pi} k! \left[ 2^{k} - \left. \begin{cases} 0 & \text{ if } k \text{ is odd} \\
\binom{k}{k/2} & \text{ if } k \text{ is even}  \end{cases}  \right\} \right]. \quad \co
\end{equation}
\end{note}

\medskip

The proof of Theorem \ref{theorem-2} begins with the recurrence
\begin{equation}
H_{n+1}(x) = 2xH_{n}(x) - 2nH_{n-1}(x), \quad \co
\end{equation}
\noindent
listed as entry $8.952.2$  in \cite{gradshteyn-2015a}.
Now let $n = 2m+1$ to obtain 
\begin{multline}
\label{rel-1}
\frac{H_{2m+2}(x) - H_{2m+2}(0)}{x} = \quad \co  \\
2x \left[ \frac{H_{2m+1}(x) - H_{2m+1}(0)}{x} \right] - 
2(2m+1) \left[ \frac{H_{2m}(x) - H_{2m}(0)}{x} \right],
\end{multline}
\noindent
where extra terms appearing in the previous identity, namely those with values of $H_{n}(x)$ at $x=0$, cancel each other 
using 
\begin{equation}
H_{n}(0) = \begin{cases} 
0 & \quad \textnormal{if} \,\, n \,\, \textnormal{ is odd} \\
(-1)^{n/2} \frac{n!}{(n/2)!} & \quad \textnormal{if} \,\, n \,\, \textnormal{ is even.} 
\end{cases}
\quad \co
\end{equation}
\noindent
Since $H_{2m+1}(0) = 0$,  \eqref{rel-1} becomes 
\begin{multline}
\label{rel-2}
\frac{H_{2m+2}(x) - H_{2m+2}(0)}{x} =  \\
2 H_{2m+1}(x) - 
2(2m+1) \left[ \frac{H_{2m}(x) - H_{2m}(0)}{x} \right]. \quad \co
\end{multline}
\noindent
Squaring gives 
\begin{eqnarray*}
\left[ \frac{H_{2m+2}(x) - H_{2m+2}(0)}{x} \right]^{2} & = & 
4 H_{2m+1}^{2}(x)  - 8(2m+1) H_{2m+1}(x) \left[ \frac{H_{2m}(x) - H_{2m}(0)}{x} \right] \\
& & + 4 (2m+1)^{2} \left[ \frac{H_{2m}(x) - H_{2m}(0)}{x} \right]^{2}. \nonumber  \quad \co
\end{eqnarray*}
\noindent
Now multiply by the weight $e^{-x^{2}}$ and integrate over $(-\infty, \infty)$. The integral coming from the second term vanishes 
since the bracketed expression is a polynomial of degree $2m-1$, so it is orthogonal to $H_{2m+1}(x)$. From the 
normalization \eqref{norma-her} one obtains the recurrence 
\begin{equation}
\label{recur-1}
I_{2m+2} =  \sqrt{\pi} (2m+1)! 2^{2m+3}+ 4(2m+1)^{2} I_{2m}. \quad \co
\end{equation}
\noindent
This is complemented by the initial data $I_{0} = 0$. Motivated by data generated from \eqref{recur-1}, introduce the
 auxiliary unknown 
\begin{equation}
J_{m} = \frac{I_{2m}}{\sqrt{\pi} \, 2^{2m+1} \, (2m)!} \quad \co
\end{equation}
\noindent
and reduce \eqref{recur-1} to 
\begin{equation}
\label{recur-2}
(2m+2)J_{m+1} = (2m+1)J_{m} + 1\quad \textnormal{with} \,\, J_{0} = 0. \quad \co
\end{equation}

\begin{lemma}
The solution to \eqref{recur-2} is given by 
\begin{equation}
\label{recur-3}
J_{m} =  1 - 2^{-2m} \binom{2m}{m}. \quad \co
\end{equation}
\end{lemma}
\begin{proof}
A particular solution of \eqref{recur-3} is $J_{m} = 1$. The homogeneous part 
\begin{equation}
(2m+2)J_{m+1} = (2m+1)J_{m} 
\end{equation}
\noindent
has a general solution 
\begin{equation}
J_{m}^{(h)} = 2^{-2m} \binom{2m}{m}
\end{equation}
\noindent
so that $J_{m} = 1 + \alpha J_{m}^{(h)}$ for some constant $\alpha$. The value $J_{0} = 0$ gives $\alpha = -1$. This 
completes the proof.
\end{proof}

The expression for $J_{m}$ now gives the evaluation for $I_{2m}$  in Theorem \ref{theorem-2}.

\section{The Chebyshev  case}
\label{cheby-int1}

The Chebyshev polynomials come in two flavors. The first kind $T_{n}(x)$, defined by 
\begin{equation}
T_{n}(\cos x) = \cos(nx)
\end{equation}
\noindent
form an orthogonal family on $[-1, 1]$ with respect to the weight $(1-x^{2})^{-1/2}$ and satisfy 
\begin{equation}
\int_{-1}^{1} \frac{T_{n}^{2}(x) \, dx}{\sqrt{1-x^{2}}} = \begin{cases} \pi & \quad \text{if} \,\,\, n = 0 \\ 
\frac{ \pi}{2}  & \quad \text{if} \,\,\, n \neq  0.  \quad \co
\end{cases}
\end{equation}
\noindent
These polynomials satisfy the recurrence $T_{n+1}(x) = 2xT_{n}(x) - T_{n-1}(x)$ \marginpar{\co} (for $n \geq 1$) with initial conditions 
$T_{0}(x) = 1$ and $T_{1}(x) = x$.

The Chebyshev polynomials of the second kind $U_{n}(x)$, defined by 
\begin{equation}
U_{n}(\cos x) = \frac{\sin((n+1)x)}{\sin x},
\end{equation}
\noindent
are orthogonal on $[-1,1]$ with respect to the weight $(1-x^{2})^{1/2}$ normalized by 
\begin{equation}
\int_{-1}^{1} U_{n}^{2}(x) \sqrt{1-x^{2}} \, dx = \frac{\pi}{2}.
\end{equation}
These polynomials satisfy the recurrence 
$U_{n+1}(x) = 2xU_{n}(x) - U_{n-1}(x)$ (for $n \geq 1$) with initial conditions 
$U_{0}(x) = 1$ and $U_{1}(x) = 2x$. This is the same recurrence as the one satisfied by $T_{n}(x)$.

The relevant integrals considered in the current work, in the case of  Chebyshev polynomials, are stated next.

\begin{theorem}
For $n \in \mathbb{N}$,
\begin{eqnarray*}
I_{n} & \equiv &  \int_{-1}^{1} \left[ \frac{T_{n}(x) - T_{n}(0)}{x} \right]^{2} \frac{dx}{\sqrt{1-x^{2}}} = \pi n,  \quad \co \\
J_{n} & \equiv &  \int_{-1}^{1} \left[ \frac{U_{n}(x) - U_{n}(0)}{x} \right]^{2}  \sqrt{1-x^{2}} \, dx = 
\begin{cases} \pi n & \quad \text{if} \,\, n \equiv 0 \bmod 2  \quad \co \\
 \pi (n+1) & \quad \text{if} \,\, n \equiv 1 \bmod 2.
 \end{cases}
 \nonumber
\end{eqnarray*}
  \end{theorem}
  \begin{proof}
The proof proceeds as before.  Square the recurrence for $T_{n}$ to obtain
  \begin{eqnarray}
  \left[ \frac{T_{n+1}(x) - T_{n+1}(0)}{x} \right]^{2} & = & 4T_{n}^{2}(x) \label{recu-tn1} \quad \co  \\
  & & - 4 \, T_{n}(x) \cdot \left[ \frac{T_{n-1}(x) - T_{n-1}(0)}{x} \right] \nonumber \\
  & & +  \left[ \frac{T_{n-1}(x) - T_{n-1}(0)}{x} \right]^{2}. \nonumber 
  \end{eqnarray}
  \noindent
  Now divide by $\sqrt{1-x^{2}}$ and integrate over $[-1,1]$. The integral coming from the second line above vanishes 
  since $(T_{n-1}(x)-T_{n-1}(0))/x$ is a polynomial of degree $n-2$ and thus orthogonal to $T_{n}(x)$. With the 
  notation 
  \begin{equation}
  a_{n} = \int_{-1}^{1} \left[ \frac{T_{n}(x)-T_{n}(0)}{x} \right]^{2} \, \frac{dx}{\sqrt{1-x^{2}}}
  \end{equation}
  \noindent
  and the normalization for $T_{n}(x)$, the recurrence \eqref{recu-tn1} yields 
  \begin{equation}
  a_{n+1} = 2\pi + a_{n-1} \quad \text{ for } n > 0. \quad \co
  \end{equation}
  \noindent
  The initial conditions $a_{1} = \pi$ and $a_{2} = 2 \pi$ lead to  the result. The proof for $U_{n}(x)$ is similar.
  \end{proof}

\section{The Laguerre case}
\label{lague-int1}

The Laguerre polynomials $L_{n}(x)$ form an orthogonal sequence on $[0, \infty)$ with the weight $w(x) = e^{-x}$. The 
explicit expression 
\begin{equation}
\label{lague-1}
L_{n}(x) = \sum_{k=0}^{n} \frac{(-1)^{k}}{k!} \binom{n}{k} x^{k} \quad \co
\end{equation}
\noindent
may be found as the special case $\alpha= 0$ of entry $8.970.1$ in  \cite{gradshteyn-2015a}. 

\smallskip 

The goal of this section is dictated by  the next theorem.

\begin{theorem}
For $n \in \mathbb{N}$, let $L_{n}(x)$ be the Laguerre polynomial. Then 
\begin{equation}
\int_{0}^{\infty} \left[ \frac{L_{n}(x) - L_{n}(0)}{x} \right]^{2} e^{-x} \, dx = 2n - H_{n} \quad \co
\end{equation}
\noindent
where $H_{n} = 1 + \tfrac{1}{2} + \cdots + \tfrac{1}{n}$ is the harmonic number.
\end{theorem}

\begin{proof}
Square \eqref{lague-1} to produce
\begin{equation}
\left[ \frac{L_{n}(x) - L_{n}(0)}{x} \right]^{2} = 
\sum_{k=1}^{n} \sum_{j=1}^{n} \binom{n}{k} \binom{n}{j} \frac{(-1)^{k+j}}{k! j!} x^{k+j-2}.
\quad \co
\end{equation}
\noindent
Now integrate with respect to the weight $e^{-x}$ and with  the notation 
\begin{equation}
\label{notation-b}
B_{n,k} = \sum_{j=1}^{n} \frac{(-1)^{j}}{j} \binom{n}{j} \binom{k+j-2}{k-1}
\end{equation}
write
\begin{equation}
\label{main-1}
\int_{0}^{\infty} \left[ \frac{L_{n}(x) - L_{n}(0)}{x} \right]^{2}  e^{-x} \, dx = 
\sum_{k=1}^{n} \frac{(-1)^{k}}{k} \binom{n}{k} B_{n,k}.  \quad \co
\end{equation}

The next step is to evaluate $B_{n,k}$.

\begin{lemma}
\label{lemma-1}
For $n \in \mathbb{N}$ and $H_{n} $ the harmonic number, 
\begin{equation}
\label{lhs-eq1}
\sum_{k=2}^{n} \frac{(-1)^{k}}{k-1} \binom{n}{k} = n(H_{n}-1). \quad \co
\end{equation}
\end{lemma}
\begin{proof}
Let $\alpha_{n}$ be the left-hand side of \eqref{lhs-eq1}. Then 
\begin{eqnarray}
\alpha_{n} - \alpha_{n-1} & = & \sum_{k=2}^{n} \frac{(-1)^{k}}{k-1} \binom{n}{k} - \sum_{k=2}^{n-1} \frac{(-1)^{k}}{k-1} 
\binom{n-1}{k} \quad \co  \\
& = & \frac{(-1)^{n}}{n-1}  + \sum_{k=2}^{n-1} \frac{(-1)^{k}}{k-1} \left[ \binom{n}{k} - \binom{n-1}{k} \right] \nonumber 
\quad \co \\
& = & \sum_{\ell=1}^{n-1} \frac{(-1)^{\ell+1}}{\ell} \binom{n-1}{\ell}.  \quad \co \nonumber 
\end{eqnarray}
\noindent
Denote this last sum by $T$ and observe that 
\begin{equation}
T = \int_{0}^{1} \frac{1 - (1-s)^{n-1}}{s} \, ds. \quad \co 
\end{equation}
\noindent
The change of variables $t = 1-s$ and expanding the new integrand gives $T = H_{n-1}$. Thus, $\alpha_{n} - 
\alpha_{n-1} = H_{n-1}$.  A direct computation shows that the right-hand side of \eqref{lhs-eq1} satisfies the same 
difference equation and since the initial values match, the identity \eqref{lhs-eq1} holds for all $n \in \mathbb{N}$. 
\end{proof}

\begin{lemma}
\label{lemma-2}
For $n \in \mathbb{N}$ and $H_{n}$ the harmonic number, 
\begin{equation}
\label{lhs-eq2}
\sum_{k=2}^{n} \frac{(-1)^{k}}{k} \binom{n}{k} = n - H_{n}. \quad \co
\end{equation}
\end{lemma}
\begin{proof}
Let $b_{n}$ be the left-hand side of \eqref{lhs-eq2}. Then 
\begin{eqnarray}
b_{n} - b_{n-1} & = & \sum_{k=2}^{n} \frac{(-1)^{k}}{k} \binom{n}{k} -  \sum_{k=2}^{n-1} \frac{(-1)^{k}}{k} \binom{n-1}{k} 
\quad \co  \\
& = & \frac{(-1)^{n}}{n} + \sum_{k=2}^{n-1} \frac{(-1)^{k}}{k} \left[ \binom{n}{k} - \binom{n-1}{k} \right] \nonumber
\quad \co  \\ 
& = & \sum_{k=2}^{n} \frac{(-1)^{k}}{k} \binom{n-1}{k-1}.  \nonumber  \quad \co
\end{eqnarray}
\noindent
Now use $\begin{displaystyle} \frac{1}{k} \binom{n-1}{k-1}  = \frac{1}{n} \binom{n}{k} \end{displaystyle}$ to obtain 
$\begin{displaystyle}b_{n} - b_{n-1}  = \frac{1}{n} \sum_{k=2}^{n} (-1)^{k} \binom{n}{k} = 1 - \frac{1}{n}\end{displaystyle}$
using the fact that the alternating sum of binomial coefficients vanish. A direct computation shows that $n-H_{n}$ 
satisfies the same difference equation and has the same initial condition as $b_{n}$. This proves the identity.
\end{proof}

\begin{lemma}
\label{lemma-3}
For $n \in \mathbb{N}$ and with the notation \eqref{notation-b}, we have $B_{n,1} = -H_{n}$. \quad \co
\end{lemma}
\begin{proof}
This follows directly from 
\begin{equation}
B_{n,1} = \sum_{j=1}^{n} \frac{(-1)^{j}}{j} \binom{n}{j} = b_{n} - n \quad \co 
\end{equation}
\noindent
and the result of Lemma \ref{lemma-2}.
\end{proof}

\begin{lemma}
\label{lemma-4}
For $2 \leq k \leq n$ and with the notation \eqref{notation-b}, we have 
$B_{n,k} = - \frac{1}{k-1}.$
\end{lemma}
\begin{proof}
Observe that 
\begin{equation}
(k-1)B_{n,k} = \sum_{j=1}^{n} (-1)^{j} \binom{n}{j} \binom{k+j-2}{j}. \quad \co
\end{equation}
\noindent
Replace $k$ by $r + 2$ and observe that the result to be proven is equivalent to the identity
\begin{equation}
\label{final-1}
\sum_{j=0}^{n} (-1)^{j} \binom{n}{j} \binom{r+j}{j} = 0, \quad \textnormal{for} \,\, 0 \leq r \leq n-1 \, 
\textnormal{and} \,\, n \geq 1. \quad \co
\end{equation}
\noindent
Define 
\begin{equation}
f(x) = \sum_{j=0}^{n} \binom{n}{j} \binom{r+j}{j} x^{j}.
\end{equation}
\noindent
Then \eqref{final-1} is the same as  $f(-1)=0$. This is established next. The expression
\begin{equation}
f(x) = \pFq21{r+1 \quad -n}{1}{-x} \quad \co
\end{equation}
\noindent
is  checked directly by writing 
\begin{equation}
\pFq21{r+1 \quad -n}{1}{-x} = \sum_{\ell=0}^{\infty} \frac{(r+1)_{\ell} (-n)_{\ell}}{(1)_{\ell} \, \ell!} (-x)^{\ell}
\end{equation}
\noindent
and simplifying  the series using 
\begin{equation*}
(r+1)_{\ell} = \frac{(r + \ell)!}{r!}, \,\,  (1)_{\ell} = \ell! \quad \textnormal{and}  \,\,\, 
(-n)_{\ell} = \begin{cases} 0 & \text{if } \ell  >  n \\ (-1)^{\ell} \frac{n!}{(n-\ell)!} & \text{if } \ell \leq n. 
\end{cases}
\quad \co
\end{equation*}
\noindent
The value 
\begin{equation}
f(-1) = \pFq21{r+1 \quad -n}{1}{1}
\end{equation}
\noindent
is seen to vanish using Gauss' evaluation (see entry $9.122.1$ in \cite{gradshteyn-2015a} and also  Theorem 
$2.2.2$ in  \cite{andrews-1999a})
\begin{equation}
\pFq21{a \quad b}{c}{1} = \frac{\Gamma(c) \Gamma(c-a-b)}{\Gamma(c-a) \Gamma(c-b)}.
\quad \co
\end{equation}
\noindent
The proof of Lemma \ref{lemma-4} is complete.
\end{proof}

The formula \eqref{main-1} and the values of $B_{n,k}$ now produce 
\begin{eqnarray*}
\int_{0}^{\infty} \left[ \frac{L_{n}(x) - L_{n}(0)}{x} \right]^{2} \, e^{-x} \, dx & = & 
\sum_{k=1}^{n} \frac{(-1)^{k}}{k} \binom{n}{k} B_{n,k}  \quad \co \\
& = & 
-n B_{n,1} + \sum_{k=2}^{n} \frac{(-1)^{k}}{k} \binom{n}{k} B_{n,k} \nonumber  \quad \co  \\
& = & nH_{n} - \sum_{k=2}^{n} \frac{(-1)^{k}}{k(k-1)} \binom{n}{k} \nonumber  \quad \co\\
& = & nH_{n} - \left( \alpha_{n} - b_{n} \right) \nonumber \quad \co \\
& = & 2n -H_{n}, \nonumber \quad \co
\end{eqnarray*}
\noindent 
as claimed. 
\end{proof}

\begin{note}
The identity \eqref{final-1} may also be obtained from two representations of the Jacobi polynomial
\begin{eqnarray*}
P_{n}^{(\alpha, \beta)}(x) & = & \frac{(-1)^{n}}{2^{n} n!} (1-x)^{-\alpha} (1+x)^{-\beta} 
\frac{d^{n}}{dx^{n}} \left( (1-x)^{n+\alpha} (1+x)^{n+ \beta} \right) \\
& = & \sum_{j} \binom{n+\alpha+\beta+j}{j} \binom{n+\alpha}{n-j} \left( \frac{x-1}{2} \right)^{j}. \nonumber 
\end{eqnarray*}
\noindent
Choose $\alpha = 0, \, \beta = r-n$ and $x = -1$. Then \eqref{final-1} is recovered from $P_{n}^{(0, n-r)}(-1) = 0$.
\end{note}

\section{The Gegenbauer case}
\label{gegenb-int1}

The Gegenbauer polynomial $C_{n}^{(a)}(x)$ is defined by the explicit expression
\begin{equation}
C_{n}^{(a)}(x) = \sum_{k=0}^{\left\lfloor \tfrac{n}{2} \right\rfloor} (-1)^{k}  
\frac{\Gamma(n-k+a)}{\Gamma(a) k! (n-2k)!} (2x)^{n-2k}, \quad \co
\end{equation}
\noindent
or by the  recurrence 
\begin{equation}
C_{n}^{(a)}(x) = \frac{1}{n} 
\left[ 2x(n+a-1) C_{n-1}^{(a)}(x) - (n+2a-2) C_{n-2}^{(a)}(x) \right], \quad \co
\end{equation}
\noindent
with initial conditions $C_{0}^{(a)}(x) = 1$ and $C_{1}^{(a)}(x) = 2ax$. The Gegenbauer polynomials are orthogonal 
in $[-1, 1]$ with 
respect to the weight $(1-x^{2})^{a-1/2}$ and are normalized by 
\begin{equation}
\int_{-1}^{1} \left[ C_{n}^{a}(x) \right]^{2} (1-x^{2})^{a-1/2} \, dx = \frac{\pi \Gamma(n+2 a)}{2^{2a-1} n! (n+a) 
\Gamma^{2}(a) }. \quad \co
\end{equation}

The integral discussed in this section  is given by 
\begin{equation}
\gamma_{n}(a) = \int_{-1}^{1} \left[ \frac{C_{n}^{a}(x) - C_{n}^{a}(0)}{x} \right]^{2} (1-x^{2})^{a-1/2} \, dx.
\end{equation}
\noindent
Proceeding as in the previous sections one may obtain a recurrence for these integrals. 

\begin{lemma}
The integrals $\gamma_{n}(a)$ satisfy the recurrence
\begin{multline}
\label{recu-gamma}
\gamma_{n}(a) = \frac{\pi (n+a-1)\Gamma(n-1+2a)}{2^{2a-3} n \, n! \,  \Gamma^{2}(a)} + \frac{(n+2a-2)^{2}}{n^{2}} \gamma_{n-2}(a) \quad \co
\end{multline}
\noindent
with initial conditions 
\begin{equation}
\gamma_{0}(a) = 0 \quad  \text{and} \quad  \gamma_{1}(a) = \frac{4 a \sqrt{\pi}\,  \Gamma(a + \tfrac{1}{2} ) }{\Gamma(a)}.
\quad \co 
\end{equation}
\end{lemma}

From the recurrence \eqref{recu-gamma}, the value $\gamma_{2m+1}(a)$ is easy to realize.  Let $n = 2m+1$ and 
write the recurrence as 
\begin{multline}
\label{recu-odd}
\gamma_{2m+1}(a) = \frac{\pi (2m+a)\Gamma(2m+2a)}{2^{2a-3} (2m+1) \, (2m+1)! \, 
 \Gamma^{2}(a)} + \frac{(2m+2a-1)^{2}}{(2m+1)^{2}} \gamma_{2m-1}(a). \quad \co
\end{multline}

\begin{proposition}
For $m \in \mathbb{N}$ odd, the integral $\gamma_{m}(a)$ is given by 
\begin{equation}
\gamma_{m}(a) = \frac{\pi  \Gamma(2a+m)}{2^{2a-2} m!  \, \Gamma^{2}(a)}. \quad \co
\end{equation}
\end{proposition}
\begin{proof}
Introduce the new variable $u_{m}$ by 
\begin{equation}
u_{m}(a) = \left(  \frac{\pi  \Gamma(2a+2m+1)}{2^{2a-2} (2m+1)!  \, \Gamma^{2}(a)} \right)^{-1} \gamma_{2m+1}(a),
\quad \co
\end{equation}
\noindent
and convert  \eqref{recu-odd}  into
\begin{equation}
(a+m) (2m+1)u_{m}(a) = 2m+a + m(2m+2a-1)u_{m-1}(a). \quad \co
\end{equation}
\noindent
An inductive argument now shows that $u_{m}(a) \equiv 1$, proving the assertion. 
\end{proof}

The next results present the evaluation of $\gamma_{n}(a)$ for $n$ even. The recurrence \eqref{recu-gamma} is now 
used to produce data, motivating  the scaling 
\begin{equation}
\R_{n}(a) = \frac{\Gamma(a) (2n)!}{4 \sqrt{\pi} \Gamma \left( a + \tfrac{1}{2} \right) \, (a)_{n}} \gamma_{2n}(a),
\quad \co
\end{equation}
\noindent
where $(a)_{n}$ is the Pochhammer symbol. Then \eqref{recu-gamma} yields:

\begin{proposition}
\label{recu-poly1}
The function $\R_{n}(a)$ satisfies the recurrence
\begin{equation*}
\R_{n}(a) = \frac{2}{n} (a+n-1)(2n-1) \R_{n-1}(a) + \frac{(2n+a-1)}{n} 2^{2n-2} \left( a+ \tfrac{1}{2} \right)_{n-1},
\quad \co
\end{equation*}
\noindent
with $\R_{0}(a)= 0$. Therefore  $\R_{n}(a)$ is a polynomial in $a$ of degree $n$.
\end{proposition}

\begin{corollary}
\label{coeff-pos}
The coefficients of $\R_{n}(a)$ are positive. 
\end{corollary}

A second recurrence for $\R_{n}(a)$ is obtained by eliminating the Pochhammer symbol in the result of 
Proposition \ref{recu-poly1}. 

\begin{proposition}
\label{recu-poly2}
The polynomials $\R_{n}(a)$ satisfy the recurrence 
\begin{eqnarray*}
n(2n+a-3)\R_{n}(a) & = & 2(2n+a-2)(4n^2+4an-8n-3a+3)\R_{n-1}(a)  \\
& & -4(a+n-2)(2n-3)(2a+2n-3)(2n+a-1)\R_{n-2}(a), \nonumber 
\end{eqnarray*}
\noindent
with initial conditions $\R_{0}(a) = 0$ and $\R_{1}(a) = a+1$.
\end{proposition}

The next result gives a zero for the polynomial $\R_{n}(a)$. From the proof one obtains alternative 
expressions  for this polynomial. In turn, these produce a simpler proof of the theorem below. The 
original proof presented here also has pedagogical interest.

\begin{theorem}
\label{zero-proof1}
The polynomial $\R_{n}$ vanishes at $a = -n$.
\end{theorem}
\begin{proof}
Introduce the function $f_{n}(a) = 2(a+n-1)(2n-1)/n$ and divide the recurrence for $\R_{n}$ by 
\begin{equation}
\label{prod-j}
\prod_{j=1}^{n} f_{j}(a) = \frac{(n+a-1)!}{(a-1)!} \binom{2n}{n} \quad \co
\end{equation}
\noindent
with $a! = \Gamma(a+1)$ for non-integer $a$.  This yields 
\begin{equation}
A_{n}(a) - A_{n-1}(a) = \frac{(2n+a-1)}{na}\frac{ \binom{2n+2a-2}{n+a-1} }{\binom{2n}{n} \binom{2a}{a}}, \quad \co
\end{equation}
\noindent
with $A_{n}(a)$ being the quotient of $\R_{n}(a)$ by the product in \eqref{prod-j}. Summing over $n$ gives 
\begin{equation}
\R_{n}(a) = \frac{(n+a-1)! \binom{2n}{n}}{a! \binom{2a}{a} }
\sum_{k=1}^{n} \frac{(2k+a-1)}{k \binom{2k}{k}} \binom{2k+2a-2}{k+a-1}. \quad \co
\end{equation}
\noindent
This can be written as 
\begin{equation}
\R_{n}(a) = (n-1)! \binom{a+n-1}{n-1} \binom{2n}{n} 
\sum_{k=1}^{n} \frac{(2k+a-1)a^{2}}{2k^{2}(2k-1)}  \frac{\binom{2a+2k-2}{2k-2}}{\binom{a-1+k}{k}^{2}} \quad \co
\end{equation}
\noindent
and using $\binom{-x+y}{y} = (-1)^{y} \binom{x-1}{y}$ gives
\begin{equation*}
\R_{n}(a) = \frac{(-1)^{n-1}}{2} (n-1)!  \binom{-a-1}{n-1} \binom{2n}{n} a^{2} 
\sum_{k=1}^{n} \frac{(2k+a-1)}{k^{2}(2k-1)} \frac{\binom{-2a-1}{2k-2}}{\binom{-a}{k}^{2}}. \quad \co
\end{equation*}
\noindent
Using this form of $\R_{n}(a)$ it follows that $\R_{n}(-n) = 0$ is equivalent to the identity
\begin{equation}
\label{sum-1}
\sum_{k=1}^{n} \frac{(2k-1-n)}{k^{2}(2k-1)} \frac{\binom{2n-1}{2k-2}}{\binom{n}{k}^{2}} = 0 \quad \co
\end{equation}
\noindent
whose  proof is based on the automatic methods developed by H.~Wilf and D.~Zeilberger \cite{petkovsek-1996a}.  To 
this end, denote the summand in \eqref{sum-1} by
\begin{equation}
F(n,k) =  \frac{(2k-1-n)}{k^{2}(2k-1)} \frac{\binom{2n-1}{2k-2}}{\binom{n}{k}^{2}}
\end{equation}
\noindent
and then using the  WZ-methodology gives the companion function 
\begin{equation}
G(n,k) = \frac{(k-n-1) \binom{2n-1}{2k-2}}{k^{2} \binom{n}{k}^{2}}
\end{equation}
\noindent
to the effect that 
\begin{equation}
F(n,k) = G(n,k+1)-G(n,k). \quad \co
\end{equation}
\noindent
Now sum from $k=1$ to $n$ and use the values $G(n,n+1) = G(n,1) = -1/n$ to obtain the result.
\end{proof}

\begin{note}
\label{note-expkn}
The alternative expression for $\R_{n}(a)$ given below follow from  elementary manipulations of those appearing in the previous proof. The 
formula is written in terms of factorials (with the usual interpretation  \newline $b! = \Gamma(b+1)$ for $b \not \in \mathbb{N}$). The formula is 
\begin{equation}
\label{kn-one}
\R_{n}(a) = \frac{1}{8} \binom{2n}{n} \frac{(a+n-1)!}{\left( a - \tfrac{1}{2} \right)!} 
\sum_{k=1}^{n} \frac{(2k+a-1)}{(2k-1)} \frac{1}{\binom{2k-2}{k-1}} 
\frac{\left( a + k - \tfrac{3}{2} \right)!}{(a+k-1)!} 2^{2k}. \quad \co
\end{equation}
\end{note}

A simpler expression for $\R_{n}(a)$ is given next. 

\begin{theorem}
For $n \geq 1$, the polynomial $X_{n}$ is given by 
\begin{eqnarray}
\R_{n}(a)  & = &   2^{2n-1} \left( a + \tfrac{1}{2} \right)_{n} - \binom{2n-1}{n-1} (a)_{n} \label{nice-xn} \\
& = & 2^{n-1} \prod_{k=0}^{n-1} (2a+2k+1) - \binom{2n-1}{n-1} \prod_{k=0}^{n-1} (a+k). \nonumber 
\end{eqnarray}
\end{theorem}
\begin{proof}
The formula \eqref{kn-one} is now written in term of Pochhammer symbols as 
\begin{equation}
\label{kn-two}
\R_{n}(a) = \frac{1}{8} \binom{2n}{n} \sum_{k=1}^{n} \frac{2^{2k}}{(2k-1) \binom{2k-2}{k-1}} 
\left[ (2k+a-1) \left(a + \tfrac{1}{2} \right)_{k-1} \left( a+k \right)_{n-k} \right].
\quad \co
\end{equation}
A symbolic evaluation of \eqref{kn-two}, using \texttt{Mathematica}, leads to \eqref{nice-xn}. To prove this assertion, simply verify 
that the right-hand side 
satisfies  \eqref{recu-poly1} for $n \geq 1$ and that both sides give $a+1$ at $n=1$.
\end{proof}

\medskip 

Some elementary properties of the polynomial $\R_{n}(a)$ can be obtained from \eqref{nice-xn}. The first class of results deal with 
the coefficients. 

\begin{property}
The leading coefficient of $\R_{n}(a)$ is $\begin{displaystyle}  2^{2n-1} - \frac{1}{2} \binom{2n}{n}.
\quad \co
\end{displaystyle}$
\end{property}

\begin{property}
\label{value-zero}
The constant term in $\R_{n}(a)$ is $\begin{displaystyle} \frac{(2n)!}{2n!}. \end{displaystyle}$
\quad \co
\end{property}

\begin{property}
The expansion
\begin{equation}
(x)_{n} = \sum_{j=0}^{n} (-1)^{n-j} s(n,j) x^{j},
\end{equation}
\noindent
where $s(n,j)$ are the Stirling numbers of the first kind is now replaced in \eqref{kn-two} to produce an expression for 
the coefficients in 
\begin{equation}
\R_{n}(a) = \sum_{r=0}^{n} \rho_{n,r} a^{r} 
\end{equation}
\noindent
in the form 
\begin{multline}
\label{form-rho}
\quad \co \\
\rho_{n,r} = 
(-1)^{n} 2^{n+1-r} \left[ \sum_{j=r}^{n} (-1)^{j}2^{n-j}  \binom{j}{r} s(n,j)  - 
(-1)^{r} 2^{n} \binom{2n-1}{n-1} s(n,r) \right]
\end{multline}
\noindent
Therefore the coefficients $\rho_{n,r}$ is an integer.
\end{property}

\begin{corollary}
The coefficients $\rho_{n,r}$ are positive integers.
\end{corollary}
\begin{proof}
This follows from \eqref{form-rho} and Corollary \ref{coeff-pos}.
\end{proof}

The second type of results relate to the values of $\R_{n}$. The first statement is another proof of Theorem \ref{zero-proof1}.

\begin{property}
\label{zero-proof2}
The polynomial $\R_{n}(a)$ vanishes at $a = -n$.
\end{property}
\begin{proof}
This follows  from the values 
\begin{equation*}
\left( - n + \tfrac{1}{2} \right)_{n} = (-1)^{n} \frac{(2n)!}{n! \, 2^{2n}} \quad \textnormal{and} \quad 
(-n)_{n} = (-1)^{n} n!,
\end{equation*}
\noindent
and \eqref{nice-xn}.
\end{proof}

\begin{property}
\label{pro-chris}
For $n \in \mathbb{N}$ and $1-n \leq k \leq -1$, we have
\begin{equation}
\R_{n}(-k) = (-1)^{k} 2^{2n-1} \left( \tfrac{1}{2} \right)_{k} \left( \tfrac{1}{2} \right)_{n-k}.
\end{equation}
\noindent
and they satisfy the symmetry condition 
\begin{equation}
\R_{n}(k-n) = (-1)^{n} \R_{n}(-k), \quad \text{for } 1 \leq k \leq n-1.
\end{equation}
\end{property}
\begin{proof}
This follows directly from \eqref{nice-xn} and the identities 
\begin{equation} 
\left( k-n+\tfrac{1}{2} \right)_{n} = (-1)^{n} \left( \tfrac{1}{2}-k \right)_{n} \quad \text{ and } 
\quad (k-n)_{n} = (-1)^{n} (-k)_{n}.
\end{equation}
\end{proof}

\begin{property}
All zeros of $\R_{n}(a)$ are real and negative.
\end{property}
\begin{proof}
The value $\R_{n}(0)$, given in Property \ref{value-zero}, is positive. On the other hand, Property \ref{pro-chris} shows that 
$\R_{n}(-1)<0$.  Therefore, there is a real zero of $\R_{n}(a)$ in the interval $(-1,0)$. 
Property \ref{pro-chris} shows that $\R_{n}(-1)<0$.  The signs of $\R_{n}(j)$ alternate for $j=-1, \, -2, \cdots, -n+1$ giving a 
zero in  the open  interval $(j-1,j)$ as shown before for $j=0$. This accounts for $n-1$ real negative zeros. Property 
\ref{zero-proof2} shows that there is one more zero exactly at 
$a=-n$, for a total of $n$.  Since $\R_{n}(a)$ is of degree $n$, these are all of them.
\end{proof}


The next statement looks at the polynomial $\R_{n}(a)$ modulo $n$, when $n$ is prime.

\begin{property}
For $q$ prime 
\begin{equation}
\R_{q}(a) \equiv a^{q} - a \bmod q.
\end{equation}
\end{property}
\begin{proof}
From $2^{q-1} \equiv 1 \bmod q$, we obtain 
\begin{equation}
2^{2q-1} \left( a + \tfrac{1}{2} \right)_{q}  \equiv  (2a+1)(2a+3) \cdots (2a+2q-1) \bmod q.
\end{equation}
\noindent
Wolstelholme theorem gives 
\begin{equation}
\binom{2q-1}{q-1} \equiv 1 \bmod q
\end{equation}
\noindent
so that 
\begin{equation}
\R_{q}(a) \equiv (2a+1)(2a+3) \cdots (2a+2q-1) - a(a+1) \cdots (a+q-1) \bmod q.
\end{equation}
\noindent
The lists  $\{ 1, \, 3, \, \cdots, 2q-1 \}$ and $\{ 0, \, 2, \,\cdots, 2q-2 \}$ are both the same list as 
$\{1, \cdots, q \}$ modulo $q$. Therefore 
\begin{eqnarray}
\R_{q}(a) & \equiv  & \prod_{j=1}^{q} (2a+2j-1) - \prod_{j=1}^{q}(a+j) \\
& \equiv &  \prod_{j=1}^{q} (2a+j) -  \prod_{j=1}^{q}(a+j) \nonumber \\
& \equiv &  \prod_{j=1}^{q} (2a+2j) -  \prod_{j=1}^{q}(a+j) \nonumber \\
& \equiv & (2^{q}-1) \prod_{j=1}^{q}(a+j) \nonumber \\
& \equiv & \prod_{j=1}^{q} (a+j),  \nonumber
\end{eqnarray}
\noindent
all congruences being taken modulo $q$. In the finite field $\mathbb{F}_{q}$, the polynomial $f(a) = a^q-a$ factors 
completely with roots $1, \, 2, \cdots, q$. Hence 
\begin{equation}
a^{q}-a \equiv \prod_{k=1}^{q} (a-k) \equiv \prod_{k=1}^{q} (a+k) \bmod q.
\end{equation}
\noindent 
This completes the proof.
\end{proof}

\noindent
We conclude the discussion on the polynomials $\R_{n}(a)$ with a question  based on extensive \texttt{Mathematica} computations: \\


\noindent
\texttt{Problem}. Define $Z_{n}(a) = \R_{n}(a)/(a+n)$. Then the family of polynomials $\{ Z_{n}(a) \}_{n \geq 1}$ have the interlacing property; that is, the
 roots of $Z_{n}$ interlace those of $Z_{n-1}$. \\
 

\noindent
\textbf{Acknowledgement}. The author wish to thank C. Koutschan for providing the result in Property \ref{pro-chris}.


\end{document}